\documentclass{amsart}[12pt]
\usepackage{epsfig}
\usepackage{bbm, enumitem}
\usepackage{bm}
\usepackage[english]{babel}
\usepackage{amssymb}
\usepackage{amsmath}
\usepackage{graphicx}
\usepackage{datetime}
\usepackage{hyperref}
\usepackage{cleveref}
\newcommand{\A}{\mathcal{A}}
\newcommand{\C}{\mathbb{C}}

\newcommand{\F}{\mathcal{F}}

\newcommand{\E}{\mathbb{E}}
\newcommand{\1}{\mathbb{1}}

\newcommand{\Z}{\mathbb{Z}}
\newcommand{\N}{\mathbb{N}}

\newcommand{\T}{\mathbb{T}}

\newcommand{\rea}{\text{Re}}
\newcommand{\ds}{\displaystyle}

\newtheorem*{thank}{\ \ \ \textbf{Acknowledgment}}
\newcounter{tictac}

\def\1{\,\rlap{\mbox{\small\rm 1}}\kern.15em 1}

\def\build#1_#2^#3{\mathrel{\mathop{\kern 0pt#1}\limits_{#2}^{#3}}}
\def\tend#1#2{\build\hbox to 12mm{\rightarrowfill}_{#1\rightarrow #2}^{ }}

\def\converge#1#2#3#4{\build\hbox to
#1mm{\rightarrowfill}_{#2\rightarrow #3}^{\hbox{\scriptsize #4}}}

\newcommand{\printdate}{\today}
\theoremstyle{definition}

\newtheorem{thm}{Theorem}[section]

\newtheorem{Prop}[thm]{Proposition}

\newtheorem{rem}[thm]{Remark}

\newcommand{\cb}{{\mathcal B}}
\newcommand{\ca}{{\mathcal A}}

\newcommand{\bmu}{\bm \mu}
\newcommand{\bmnu}{\bm \nu}
\newcommand{\bml}{\bm \lambda}

\newtheorem*{xconj}{Conjecture}
\def\En#1{(\lfloor #1 \rfloor)}

\newcommand{\beq}{\begin{equation}}
\newcommand{\eeq}{\end{equation}}

\begin{document}
\title[Ergodic bilinear averages]{On the homogeneous ergodic bilinear averages with M\"{o}bius and Liouville weights}
\author{E.\ H. \ {el} Abdalaoui}

\address{ University of Rouen Normandy \\
  LMRS UMR 60 85 CNRS\\
Avenue de l'universit\'e, BP.12 \\
76801 Saint Etienne du Rouvray - France . \\}
\email{elhoucein.elabdalaoui@univ-rouen.fr}


\keywords{multilinear ergodic averages, Furstenberg's problem of a.e. convergence, Liouville function, M\"{o}bius function, Birkhoff ergodic theorem,
Bourgain's double recurrence theorem, Zhan's estimation, Davenport-Hua's estimation, Sarnak's conjecture.\\
%$$
%{}^{\flat}\includegraphics[scale=0.1]{dali.eps}
% $$
\printdate }
 
\begin{abstract}It is shown that the homogeneous ergodic bilinear averages with M\"{o}bius or Liouville weight converge almost surely to zero, that is, 
if $T$ is a map acting on a probability space $(X,\ca,\nu)$, and $a,b \in \Z$,
%$T_1,T_2$ are the powers of $T$, 
then for any $f,g \in L^2(X)$, for almost all $x \in X$,
$$
\frac1{N}\sum_{n=1}^{N}\bmnu(n) f(T^{an}x)g(T^{bn}x) \tend{N}{+\infty}0,
$$

%and
%$$
%\frac1{N}\sum_{n=1}^{N}\bmu(n) f(T^{np}x)g(T^{nq}x) \tend{N}{+\infty}0
%$$
\noindent where $\bmnu$ is the Liouville function or the M\"{o}bius function. We further obtain that
the convergence almost everywhere holds for the short interval with the help of Zhan's estimation.
 %Our proof yields also  an accessible proof of Bourgain's double recurrence theorem.
Our proof yields also a simple proof of Bourgain's double recurrence theorem.
Moreover,
we establish that  if $T$ is weakly mixing and its restriction to its Pinsker algebra has singular spectrum, then for any integer $k \geq 1$, for any $f_j \in L^{\infty}(X),$ $j=1,\cdots,k$, for almost all $x \in X$, we have
$$\frac1{N}\sum_{n=1}^{N}\bmnu(n) \prod_{j=1}^{k}f_j(T^{nj}x)\tend{N}{+\infty}0.$$
%and
%$$\frac1{N}\sum_{n=1}^{N}\bmu(n) \prod_{j=1}^{k}f(T^{nj}x)\tend{N}{+\infty}0$$
\end{abstract}
\maketitle
\section{Introduction}
The purpose of this short note is to establish that the homogeneous ergodic bilinear averages with M\"{o}bius or Liouville weight converge almost surely to zero. Our result, 
in some sense, extend Sarnak's result which assert that the ergodic averages with M\"{o}bius or Liouville weight converge almost surely 
to zero \cite{Sarnak}. Moreover, our proof allows us to obtain a simple proof of Bourgain's double recurrence theorem \cite{Bourgain-D}.
%does not used Bourgain's double recurrence theorem \cite{Bourgain-D} but the spirit of its proof.\\

The problem of the convergence almost everywhere (a.e.) of the ergodic  multilinear averages was introduced by Furstenberg in \cite[Question 1 p.96]{Fbook}. Later, J. Bourgain proved that the homogeneous ergodic bilinear averages converges
almost surely
%provided that the maps are the powers of a given map
\cite{Bourgain-D}. Subsequently, I. Assani established that the convergence a.e. of the homogeneous ergodic multilinear averages holds %if the maps are the powers of a given map  for which its
if the restriction of the map to its Pinsker algebra has a singular spectrum. Assani's proof is based essentially on Bourgain's theorem \cite{Bourgain-D} combined with Host's joining theorem \cite{Host}. Very recently, E. H. el Abdalaoui proved that there is a subsequence for which
%To the best knowdge's of the author, the .
the convergence a.e. of the ergodic multilinear averages holds \cite{elabdal-F}.
For a recent survey on the \linebreak Furstenberg's problem on the ergodic multilinear averages, we refer to \cite{Ye1}.
Let us mention also that C. Demeter in \cite{Demeter} obtained an alternative proof of Bourgain's theorem \cite{Bourgain-D}.\\

For the almost everywhere convergence of the homogeneous ergodic bilinear averages with weight, I. Assani, D. Duncan, and R. Moore proved \`a la Wiener-Wintner 
that the exponential sequences $(e^{2\pi i nt})_{n \in \Z}$ are good weight for the homogeneous ergodic bilinear averages \cite{AssaniDM}. 
Subsequently, I. Assani and R. Moore  showed that the polynomials exponential sequences $\big(e^{2\pi i P(n)}\big)_{n \in \Z}$  are also uniformly good weights for the homogeneous ergodic bilinear averages
\cite{AssaniM}. One year later, I. Assani \cite{Assani-Nil} and P. Zorin-Kranich \cite{Zorich} proved independently that the nilsequences are
 uniformly good weights for the homogeneous ergodic bilinear averages. Their proof depend heavily on Bourgain's theorem. Let us further notice that Zorin-Kranich's proof yields that  
if the ergodic multilinear averages converges a.e. then the nilsequences are a good weight for the ergodic multilinear averages. \\

Here, our goal is to prove that the M\"{o}bius and  Liouville functions are a good weight for the homogeneous ergodic bilinear averages. Our proof follows closely Bourgain's proof \cite{Bourgain-D}. 
We thus apply Calder\'{o}n transference principal in order to establish some kind of maximal inequality. Furthermore, 
we apply Assani's result to prove that the M\"{o}bius and Liouville functions are a good weight for the homogeneous ergodic 
multilinear averages if the restriction of the map to its Pinsker algebra has singular spectrum.\\

Let us recall that Sarnak announced in his seminal paper \cite{Sarnak} that the M\"{o}bius function is a good weight in $L^2$ for the ergodic averages. In \cite{elabdalDCDS}, the authors apply Davenport's estimation combined with Etamedi's trick \cite{Etemadi} to obtain a simple proof of Sarnak's result. Therein, they proved that the M\"{o}bius function is a good weight in $L^1$ for the ergodic averages.

\section{Notations and Tools}
The Liouville function is defined for the positive integers $n$ by
$$
\bml=(-1)^{\Omega(n)},
$$
where $\Omega(n)$ is the length of the word $n$ is the alphabet of prime, that is, $\Omega(n)$ is the number of prime factors of $n$ counted with multiplicities. The M\"{o}bius function is given by
\begin{equation}\label{Mobius}
\bmu(n)= \begin{cases}
 1 {\rm {~if~}} n=1; \\
\bml(n)  {\rm {~if~}} n
{\rm {~is~the~product~of~}} r {\rm {~distinct~primes}}; \\
0  {\rm {~if~not}}
\end{cases}
\end{equation}
These two functions are of great importance in number theory since the Prime Number Theorem is equivalent to \begin{equation}\label{E:la}
\sum_{n\leq N}\bml(n)={\rm o}(N)=\sum_{n\leq N}\bmu(n).
\end{equation}
Furthermore, there is a connection between these two functions and Riemann zeta function, namely
$$
\frac1{\zeta(s)}=\sum_{n=1}^{\infty}\frac{\bmu(n)}{n^s} \text{ for any }s\in\mathbb{C}\text{ with }\rea(s)>1.
$$
Moreover, Littlewood proved that the estimate
\[
\left|\ds \sum_{n=1}^{x}\bmu(n)\right|=O\left(x^{\frac12+\varepsilon}\right)\qquad
{\rm as} \quad  x \longrightarrow +\infty,\quad \forall \varepsilon >0
\]
is equivalent to the Riemann Hypothesis (RH) (\cite[pp.315]{Titchmarsh}).\\

Here, we will need the following Davenport-Hua's estimation \cite{Da},  \cite[Theorem 10.]{Hua}: for each $A>0$, for any $k \geq 1$, we have
\begin{equation}\label{vin}
\max_{z \in \T}\left|\displaystyle\sum_{n \leq N}z^{n^k}\bml(n)\right|\leq C_A\frac{N}{\log^{A}N}\text{ for some }C_A>0.
\end{equation}
This estimate has been generalized for the short interval by T. Zhan \cite{Ztao} as follows:
for each $A>0$, for any $\varepsilon>0$, we have
\begin{equation}\label{Zhantao}
\max_{z \in \T}\left|\displaystyle\sum_{ N \leq n \leq N+M}z^n\bml(n)\right|\leq C_{A,\varepsilon}\frac{M}{\log^{A}(M)}\text{ for some }C_{A,\varepsilon}>0,
\end{equation}
provided that $M \geq N^{\frac58+\varepsilon}$.\\

The inequalities \eqref{vin} and \eqref{Zhantao} can be established also for the M\"{o}bius function by applying carefully the following identity (see for instance \cite[section 6.]{GreenT} or \cite{Batman-Chowla}):\\

$$\bml(n)=\sum_{d:d^2 | n}\bmu(\frac{n}{d^2}).$$
  
Davenport-Hua's estimation was extended by Green-Tao to the nilsequences setting. We refer to Theorem 1.1 in \cite{GreenT} for the exact estimation 
and for the definition of the nilsequences.

\vskip 0.2cm
In our setting, we consider also the ergodic multilinear averages given by
$$\frac1{N}\sum_{n=1}^{N}\prod_{i=1}^{k}f_i(T_i^nx),$$
where $k\geq 2$, $(X,\cb,\mu,T_i)_{i=1}^{k}$ are a finite family of dynamical systems where $\mu$ is a probability measure, $T_i$ are commuting  invertible  measure preserving transformations and  $f_1,f_2,\cdots,f_k$  a finite family of bounded functions. The bilinear case corresponds to $k=2$.\\

The ergodic multilinear averages is said to be homogeneous if $T_i$, $i=1,\cdots k,$ are the powers of some given map $T$.\\

For the convergence a.e., J. Bourgain proved

\begin{thm}[Bourgain's double recurrence theorem \cite{Bourgain-D}]\label{TDB}Let $(X,\ca,\mu,T)$ be an ergodic dynamical system, and $T_1,T_2$ be powers of $T$. Then, for any $f,g \in L^{\infty}(X)$, for almost all $x \in X$,
$$\frac1{N}\sum_{n=1}^{N}f(T_1^nx)g(T_2^nx)$$
converges.
\end{thm}
Applying Host's joining theorem \cite{Host} combined with Bourgain's theorem (Theorem \ref{TDB}), I. Assani proved
\begin{thm}[Assani's multilinear recurrence theorem for a map with singular spectrum \cite{Assani} ]\label{TMA}
Let $(X,\ca,\mu,T)$  be a weakly mixing dynamical system such that the
restriction of $T$ to its Pinsker algebra has singular spectrum, then, for all
positive integers $k$, for all $f_i \in L^\infty(X)$, $i=1,\cdots,k$, for almost all $x \in X$, we have
$$\frac1{N}\sum_{n=1}^{N}\prod_{i=1}^{k}f_i(T^{in}x)\tend{N}{+\infty} \prod_{i=1}^{k}\int f_i(x)d\mu.$$
\end{thm}
Theorem \ref{TMA} has been extended in the sense of the multiple retrun times theorems in \cite{AssaniIHP}. Therein, the author proved that for $\mu$ a. e. $x$ the sequence
$\Big(\prod_{i=1}^{k}1_A(T^{b_in}x)\Big)_n$ is a good universal weight for the classical pointwise convergence.   
%Finally, for any $\rho>1$, we denote by $I_\rho$ the set $\Big\{\En{\rho^n}, n \in \N \Big\}.$
\section{Some tools on the maximal ergodic inequalities and Calder\'{o}n transference principle}
We say that the sequence of complex number $(a_n)$ is good weight in $L^p(X,\mu)$, $p\geq1$ for linear case, if, for any $f \in L^p(X,\mu)$, 
the ergodic averages
$$\frac{1}{N}\sum_{j=1}^{N}a_jf(T^jx)$$
converges a.e.. We further say that the maximal ergodic inequality holds in $L^p(X,\mu)$ for the linear case with weight $(a_n)$ if, 
for any $f \in L^p(X,\mu)$, the maximal function given by 
$$M(f)(x)=\sup_{N \geq 1}\Big|\frac{1}{N}\sum_{j=1}^{N}a_jf(T^jx)\Big|$$
satisfy the weak-type inequality 
$$\lambda \mu\Big\{x~~:~~M(f)(x)>\lambda \Big\} \leq C \big\|f\big\|_p,$$
for any $\lambda>0$ with $C$ is an absolutely constant.\\

It is well known that the classical maximal ergodic inequality is equivalent to the Birkhoff ergodic theorem \cite{Garcia}.\\

The previous notions can be extended in the usual manner to the multilinear case. Let $k \geq 2$, we thus say that $(a_n)$ 
is good weight in $L^{p_i}(X,\mu)$, $p_i\geq1$, $i=1,\cdots,k$, with 
$\sum_{i=1}^{k}\frac1{p_i}=1,$ if, for any $f_i \in L^{p_i}(X,\mu)$, $i=1,\cdots,k$,
the ergodic $k$-multilinear averages   
$$\frac1{N}\sum_{j=1}^{N}a_j\prod_{i=1}^{k}f_i(T_i^jx),$$
converges a.e.. The maximal multilinear ergodic inequality is said to hold in $L^{p_i}(X,\mu)$, $p_i\geq1$, $i=1,\cdots,k$, with 
$\sum_{i=1}^{k}\frac1{p_i}=1,$ if, for any $f_i \in L^{p_i}(X,\mu)$, $i=1,\cdots,k$, the maximal function given by 
$$M(f_1,\cdots,f_k)(x)=\sup_{N \geq 1}\Big|\frac{1}{N}\sum_{j=1}^{N}a_j\prod_{i=1}^{k}f_i(T_i^jx)\Big|$$
satisfy the weak-type inequality 
$$\lambda \mu\Big\{x~~:~~M(f)(x)>\lambda \Big\} \leq C \prod_{i=1}^{k}\big\|f_i\big\|_{p_i},$$
for any $\lambda>0$ with $C$ is an absolutely constant.\\

It is not known whether the classical maximal multilinear ergodic inequality ($a_n=1$, for each $n$) holds 
for the general case $n \geq 3$. Nevertheless, we have the following Calder\'{o}n transference principal in the homogeneous case. 
\begin{Prop}\label{CalderonP} Let $(a_n)$ be a sequence of complex number and assume that for any $\phi,\psi \in \ell^2(\Z)$, we have
$$\Big\|\sup_{ N  \geq 1}\Big|\frac1{N}\sum_{n=1}^{N}a_n \phi(j+n)\psi(j-n) 
\Big|\Big\|_{\ell^1(\Z)}< C.
\big\|\phi\big\|_{\ell^2(\Z)}\big\|\psi\big\|_{\ell^2(\Z)},$$
where $C$ is an absolutely constant. Then, for any dynamical system $(X,\A,T,\mu)$, for any $f,g \in L^2(X,\mu)$, we have 
%Then, 
%for any $\rho>1$, for any $\delta>0$, there exist $N_0 \in I_\rho $ independent of 
%$f$ and $g$ such that 
$$\Big\|\sup_{N \geq 1}\Big|\frac1{N}\sum_{n=1}^{N}a_n f(T^nx)g(T^{-n}x) \Big|\Big\|_1<
C \big\|f\big\|_{2}\big\|g\big\|_{2},$$ 
\end{Prop}
We further have
%In \cite{DemeterLT}, the authors  that the the classical maximal multilinear ergodic inequality holds for the 
%homogneous case provided $\frac{1}{p_{k+1}}<3/2$.\\
\begin{Prop}\label{CalderonP2} Let $(a_n)$ be a sequence of complex number and  assume that for any $\phi,\psi \in \ell^2(\Z)$, for any $\lambda>0$, for any integer $J \geq 2$, we have
$$\Big|\Big\{1 \leq j \leq J~~:~~ \sup_{ N  \geq 1}\Big|\frac1{N}\sum_{n=1}^{N}a_n \phi(j+n)\psi(j-n) 
\Big|> \lambda \Big\}\Big|<C \frac{\big\|\phi\big\|_{\ell^2(\Z)}\big\|\psi\big\|_{\ell^2(\Z)}}{\lambda},$$
where $C$ is an absolutely constant. Then, for any dynamical system $(X,\A,T,\mu)$, for any $f,g \in L^2(X,\mu)$, we have 
%Then, 
%for any $\rho>1$, for any $\delta>0$, there exist $N_0 \in I_\rho $ independent of 
%$f$ and $g$ such that 
$$\mu\Big\{x \in X~~:~~\sup_{N \geq 1}\Big|\frac1{N}\sum_{n=1}^{N}a_n f(T^nx)g(T^{-n}x) \Big| > \lambda \Big\}<
C \frac{\big\|f\big\|_{2}.\big\|g\big\|_{2}}{\lambda}.$$ 
\end{Prop}
It is easy to check that Proposition \ref{CalderonP} and \ref{CalderonP2} hold for the homogeneous $k$-multilinear ergodic averages, for any 
$k \geq 3$. Moreover, one may state and prove the finitary version where $\Z$ is replaced by $\Z/\bar{J}\Z$ and the functions $\phi$ and 
$\psi$ with $\bar{J}$-periodic functions. We refer to Proposition 14.1 in \cite{DemeterLT} for more details.\\

In finitary setting, for $\bar{J} \geq 1, p \geq 1$ and $f$ a $\bar{J}$-periodic function , we put 
$$\Z_{\bar{J}}= \Z/\bar{J}\Z,$$
$$\E_{\Z_{\bar{J}}}(f)=\frac{1}{\bar{J}}\sum_{j=1}^{\bar{J}}f(j),$$
$$\big\|f\big\|_{\ell^p(\Z_{\bar{J}})}=\Big(\frac{1}{\bar{J}}\sum_{j=1}^{\bar{J}}|f(j)|^p\Big)^{\frac1{p}}.$$
Moreover, as customary , we will denote by $\1$ the function from $\Z$ to $\C$ defined by $x \mapsto 1$, and by $\1_{\bar{J}}$ the indicator function of the interval $[-2\bar{J},2\bar{J}]$.   We will denote also by $\1_{\bar{J}}$ the $\bar{J}$-periodic function defined by $x \in \Z_{\bar{J}} \mapsto 1.$
\section{Main results and its proof}
We start by stating our first main result.

\begin{thm}\label{Mainofmain}Let $(X,\ca,\mu,T)$ be an ergodic dynamical system, and $T_1,T_2$ be powers of $T$. Then, for any $f,g \in L^2(X)$, for almost all $x \in X$,
$$\frac1{N}\sum_{n=1}^{N}\bmnu(n) f(T_1^nx)g(T_2^nx) \tend{N}{+\infty}0,$$
where $\bmnu$ is the Liouville function or the M\"{o}bius function.
\end{thm}
Our second main result can be stated as follows:
\begin{thm}\label{MainII}Let $(X,\ca,\mu,T)$ be weakly mixing ergodic dynamical system, and  assume that the spectrum of the restriction of $T$ to its Pinsker algebra is singular. Then, for any
$f_j \in L^{\infty}(X),$ $j=1,\cdots,k$, for almost all $x \in X$, we have
$$\frac1{N}\sum_{n=1}^{N}\bmnu(n) \prod_{j=1}^{k}f_j(T^{jn}x)\tend{N}{+\infty}0,$$
where $\bmnu$ is the Liouville function or the M\"{o}bius function.
\end{thm}
For the proof of Theorem \ref{MainII}, we need the following criterion based on the results of Bourgain-Sarnak-Ziegler \cite[Theorem 2]{KBSZ}, and Katai \cite{Katai} , which in turn 
develop some ideas of Daboussi (presented in \cite{Daboussi}, \cite{DaboussiII}). Let us notice that similar results are presented by Harper in \cite{Harper} and Ramar\'e in \cite{Ramare}.

\begin{thm}[Katai-Bourgain-Sarnak-Ziegler's (KBSZ) criterion \cite{KBSZ}, \cite{Katai}]\label{KBSZ}  Let $f$ be an arithmetic bounded function and let $\bmnu$  be a bounded multiplicative function. Assume that for all sufficiently large distinct primes $p,q$ we have
$$\frac1{N}\sum_{n=1}^{N} f(np) \overline{f(nq)} \tend{N}{+\infty}0.$$
Then
$$\frac1{N}\sum_{n=1}^{N} \bmnu(n)f(n)  \tend{N}{+\infty}0.$$
\end{thm}
\begin{proof}[\textbf{Proof of Theorem \ref{MainII}}.] The proof goes by induction on $k$. We further assume that for some $i \in
\big\{1,\cdots,k\big\}$, $\ds \int f_i d\mu(x)=0$. The case $k=1$ follows from Sarnak's result \cite{Sarnak}. For $k=2$, put
$$F(n)=f_1(T_1^nx)f_2(T_2^nx),$$
where $T_1,T_2$ are the powers of $T$. Then, by Theorem \ref{TMA}, for almost all $x \in X$, for all $p \neq q$,
$$\frac1{N}\sum_{n=1}^{N} F(np) \overline{F}(nq) \tend{N}{+\infty}0,$$
Therefore, by KBSZ criterion (Theorem \ref{KBSZ}), we get, for almost all $x \in X$,
$$\frac1{N}\sum_{n=1}^{N} \bmnu(n)F(n)  \tend{N}{+\infty}0,$$
that is, for almost all $x \in X$,
$$\frac1{N}\sum_{n=1}^{N} \bmnu(n) f_1(T_1^nx)f_2(T_2^nx) \tend{N}{+\infty}0.$$
But, for any $f_1,f_2 \in L^\infty(X)$,
\begin{eqnarray*}
&&\frac1{N}\sum_{n=1}^{N} \bmnu(n) f_1(T_1^nx)f_2(T_2^nx)\\
&=& \frac1{N}\sum_{n=1}^{N} \bmnu(n) \Big(f_1-\int f_1 d\mu\Big)(T_1^nx)f_2(T_2^nx)+ 
\Big(\int f_1 d\mu(x)\Big)\frac1{N}\sum_{n=1}^{N} \bmnu(n) f_2(T_2^nx).
\end{eqnarray*}
Consequently, for any $f_1,f_2  \in L^\infty(X)$, for almost all $x \in X$,
$$\frac1{N}\sum_{n=1}^{N} \bmnu(n) f_1(T_1^nx)f_2(T_2^nx) \tend{N}{+\infty}0.$$
Now, assume that for almost all $x \in X$, and for any $\ell \leq k$, we have
$$\frac1{N}\sum_{n=1}^{N} \bmnu(n) \prod_{i=1}^{\ell}f_i(T_i^nx) \tend{N}{+\infty}0,$$
where $T_i, i=1,\cdots \ell$ are the powers of $T$. Then, by
applying again Theorem \ref{TMA}, we see that for almost all $x$,
$$\frac1{N}\sum_{n=1}^{N} F(np) \overline{F}(nq) \tend{N}{+\infty}0,$$
where
$$F(n)=\prod_{i=1}^{k+1}f_i(T_i^nx).$$
Hence, once again by  KBSZ criterion (Theorem \ref{KBSZ}), it follows that for almost all $x \in X$,
$$\frac1{N}\sum_{n=1}^{N} \bmnu(n)F(n)  \tend{N}{+\infty}0,$$
whence, for almost all $x \in X$,
$$\frac1{N}\sum_{n=1}^{N} \bmnu(n) \prod_{i=1}^{k+1}f_i(T_i^nx) \tend{N}{+\infty}0.$$
The proof of the theorem is complete.
\end{proof}
\begin{rem} Of course, our proof yields that the convergence a.e. holds for any bounded multiplicative function $\bmnu$. We deduce also that Theorem \ref{MainII} is valid for  the class of weakly mixing PID or the distal flows  with the help of the recent result of Gutman-Huang-Shao-Ye \cite{Ye1} and Huang-Shao-Ye \cite{Ye2}. 
Obviously, if the answer to Furstenberg's question \cite{Fbook} is positive then Theorem \ref{MainII} holds for the general case. Indeed, by the decomposition theorem \footnote{see for instance \cite[Proposition 3.1]{CHN}.}, for any $k \in \N$, for every $\varepsilon>0$, there  exist measurable functions  $f_{ns} ,f_z ,f_e $, such that
\begin{enumerate}[label=(\alph*)]
	\item $\|f_{\kappa}\|_{\infty} \leq 2 \|f\|_{\infty}$ with $\kappa \in \{ns,z,e\}$.
	\item $f = f_{ns} + f_z + f_e$ with  $|\|f_z\||_{k+1} = 0;~~ \|f_e\|_1 <\varepsilon;$ and
	\item for $\mu$ almost every $x \in X$, the sequence $(f_{ns}(T^nx))_{n \in \N}$ is a $k$-step nilsequence,
\end{enumerate}
where $|\|.\||_{k+1}$ is the Gowers norms. For the definition and more details on the Gowers norms, we refer to \cite{Tao-Vu}. Therefore,  by Green-Tao theorem \cite[Theorem 1.1]{GreenT},
the limit for $f_{ns}$ is zero. Now, as before, by applying DKBSZ theorem to the part $f_z$, we see that the limit is zero. Finaly, we get that the limit for $f_e$ is less than $\varepsilon.$ Whence, the limit for $f$ is less than $\varepsilon,$ since $\varepsilon>$ was arbitrary, we conclude that the limit is zero. 
\end{rem}
We move now to prove Theorem \ref{Mainofmain}. For any $\rho>1$, we will denote by $I_\rho$ the set $\Big\{\En{\rho^n}, n \in \N \Big\}.$ The maximal functions are defined by \\
$$M_{N_0,\bar{N}}(f,g)(x)=\sup_{\overset{ N_0 \leq N \leq \bar{N}}{N \in I_\rho}}\Big|\frac1{N}\sum_{n=1}^{N}\bmnu(n) f(T^{n}x) g(T^{-n}x)
-\frac1{N_0}\sum_{n=1}^{N_0}\bmnu(n) f(T^{n}x) g(T^{-n}x)\Big|,$$
$$M_{N_0}(f,g)(x)=\sup_{\overset{N \geq N_0}{N \in I_\rho}}\Big|\frac1{N}\sum_{n=1}^{N}\bmnu(n) f(T^{n}x) g(T^{-n}x)-
\frac1{N_0}\sum_{n=1}^{N_0}\bmnu(n) f(T^{n}x) g(T^{-n}x)\Big|.$$
Obviously, 
$$\lim_{\bar{N} \longrightarrow +\infty}M_{N_0,\bar{N}}(f,g)(x)=M_{N_0}(f,g)(x).$$
For the shift $\Z$-action, the maximal functions are denoted by $m_{N_0,\bar{N}}(\phi,\psi)$ and $m_{N_0}(\phi,\psi)$.\\

We start by proving the following:
\begin{thm}\label{CalderonI}For any $\rho>1$, %for any $\delta>0$, 
for any $f,g \in \ell^2(\Z)$, 
%there exist $N_0 \in I_\rho $ independent of 
%$f$ and $g$ 
for any $K \geq 1$, we have
%such that
\begin{eqnarray}\label{BourgainMaximal1} 
\sum_{k=1}^{K}
\Big\|m_{N_k,N_{k+1}}(f,g)\Big\|_{\ell^1(\Z)} < C. \sqrt{K}.
\big\|f\big\|_{\ell^2(\Z)}\big\|g\big\|_{\ell^2(\Z)},
\end{eqnarray}
where $\bmnu$ is the Liouville function or the M\"{o}bius function and $C$ is an absolutely constant which depend only on $\rho$.
\end{thm}
The classical Calder\'{o}n transference principal (see Proposition \ref{CalderonP} and \ref{CalderonP2}) allows us to obtain from Theorem \ref{CalderonI} the following:
\begin{thm}\label{MaximalI}Let $(X,\A,T,\mu)$ be an ergodic dynamical system, and let $f,g \in L^2(X,\mu)$ . Then, 
for any $\rho>1$, %for any $\delta>0$, %there exist $N_0 \in I_\rho $ independent of 
%$f$ and $g$ such that  
for any $K \geq 1$,
\begin{eqnarray}\label{BourgainMaximal2}
\sum_{k=1}^{K}\Big\|M_{N_k,N_{k+1}}(f,g)\Big|\Big\|_1< 4C .\sqrt{K}.\big\|f\big\|_{2}\big\|g\big\|_{2},
\end{eqnarray}
where $\bmnu$ is the Liouville function or the M\"{o}bius function.
\end{thm}
\noindent{}Let us give the proof of Theorem \ref{MaximalI}.

\begin{proof}[\textbf{Proof of Theorem \ref{MaximalI}}.] The proof goes, as in the proof of Proposition \ref{CalderonP} and 
\ref{CalderonP2}. Let $\bar{N} = N_{K+1}$ and $\bar{J} \gg \bar{N}$.  Put
$$\phi_x(n)=\left\{
              \begin{array}{ll}
                f(T^nx), & \hbox{if $n \in [-2\bar{N},2\bar{N}]$;} \\
                0, & \hbox{if not,}
              \end{array}
            \right.$$
and
$$\psi_x(n)=\left\{
              \begin{array}{ll}
                g(T^nx), & \hbox{if $n \in [-2\bar{N},2\bar{N}]$;} \\
                0, & \hbox{if not,}
              \end{array}
            \right.
$$
Then, by Theorem \ref{CalderonI}, we have
\begin{eqnarray*}
\sum_{k=1}^{K}\Big\|
m_{N_k,N_{k+1}}(\phi_x,\psi_x)\Big\|_{\ell^1(\Z)}<C \sqrt{K}   \big\|\phi_x\big\|_{\ell^2(\Z)} \big\|\psi_x\big\|_{\ell^2(\Z)}.
\end{eqnarray*}
We thus get
$$\sum_{k=1}^{K} \Big(
\sum_{|j| \leq \bar{N}} m_{N_k,N_{k+1}}(\phi_x,\psi_x)(j)\Big)< C \sqrt{K} \big\|\phi_x\big\|_{\ell^2(\Z)} \big\|\psi_x\big\|_{\ell^2(\Z)},$$
which can be rewritten as follows
\begin{eqnarray*}
\sum_{k=1}^{K} \sum_{|j| \leq \bar{N}} \Big(M_{N_k,N_{k+1}}(f,g)(T^jx)\Big) 
 < C \sqrt{K}\Big(\sum_{|n| \leq 2 \bar{N}}|f|^2(T^nx)\Big)^{\frac12}
\Big(\sum_{|n| \leq 2 \bar{N}}|g|^2(T^nx)\Big)^{\frac12} . 
\end{eqnarray*}
Integrating and applying H\"{o}lder inequality we obtain
$$\sum_{k=1}^{K}\Big\|M_{N_k,N_{k+1}}(f,g)\Big\|_1<4 C \sqrt{K}\big\|f\big\|_2 \big\|g\big\|_2,$$
since $T$ is measure preserving, and this finish the proof of the theorem.
\end{proof}
We proceed now to the proof of Theorem \ref{CalderonI}. Our proof follows Bourgain's arguments combined with
Davenport-Hua estimation.
\begin{proof}[\textbf{Proof of Theorem \ref{CalderonI}}.] We proceed by using the finitary method as in \cite{DemeterLT}. Let 
$\bar{J}$ be a large integer, $f,g \in \ell^2(\Z)$ and $\delta>0$. Denote by $\F$ the discrete Fourier transform on $\Z_{\bar{J}}$.  
%$\ell^2(\Z_J)$ to $L^2(\T)$, $\T$ is the circle. 
We recall that
$\F(f)(\chi)=\sum_{n \in \Z_{\bar{J}}}f(n)\chi(-n)$, for $\chi \in {\widehat{\Z_{\bar{J}}}}$, we still denote by $f$ the $\bar{J}$-periodic
function associated to $f$. 
%the $\bar{J}$-periodic function associated to $f$ is defined as follows. Let us define first the truncure function of $f$. Put
%$$
%\widetilde{f}(j)=\begin{cases}
 %                  f(j) ~~~~~\text{if}~~ 1\leq j \leq \bar{J}\\
  %                 0    ~~~~~~~~~\text{if not. }
  %                \end{cases}
  %                $$
%The $\bar{J}$-periodic function associated to $f$ is the $\bar{J}$-periodic function associated to $\widetilde{f}.$                    

For any $j \in \Z$, put
$$\bmnu_j(n)=\bmnu(j-n),~~n \in \Z.$$
Of course $\bmnu$ is extended to the negative integers $\Z_{-}$ in the usual fashion.\\

Obviously, we have
\begin{eqnarray}\label{conv}
\frac1{N}\sum_{n=1}^{N}\bmnu(n) f(j+n)g(j-n)&=& \Big(f*\Big(\frac1{N}. \bmnu_j.g.\1_{[j-N,j[} \Big)\Big)(2j),
\end{eqnarray}
where $*$ is the operation of convolution given by
$$(a*b)(j)=\sum_{n \in \Z}a(j-n)b(n),~~~~~~~~~\forall a,b \in \ell^2(\Z) \textrm{~~and~~} \forall j \in \Z.$$
Furthermore, by a standard arguments, we can rewrite \eqref{conv} as follows
\begin{eqnarray}\label{Fourier}
\frac1{N}\sum_{n=1}^{N}\bmnu(n) f(j+n)g(j-n)&=& \F^{-1}\big( \F(f).\F(G_{N,j})\big)(2j),
\end{eqnarray}
where $G_{N,j}=\frac1{N}. \bmnu_j.g.\1_{[j-N,j]}$. Therefore
\begin{eqnarray}\label{FourierII}
\Big|\frac1{N}\sum_{n=1}^{N}\bmnu(n) f(j+n)g(j-n)\Big|
&=&\Big|\int_{{\widehat{\Z_{\bar{J}}}}} \F(f)(\chi) \F(G_{N,j})(\chi) \chi(2j) d\chi\Big| \quad \quad \quad 
\quad \quad \nonumber\\
&& \leq \int_{{\widehat{\Z_{\bar{J}}}}} \big|\F(f)(\chi)\big| \big|\F(G_{N,j})(\chi)\big|d\chi
\end{eqnarray}
But, a straightforward computations gives
\begin{eqnarray}\label{FourierIII}
&&\int_{{\widehat{\Z_{\bar{J}}}}} \big|\F(f)(\chi)\big| \big|\F(G_{N,j})(\chi)\big|d\chi \nonumber\\
&=&\int_{{\widehat{\Z_{\bar{J}}}}} \Big|\sum_{n \in \Z_{\bar{J}}}f(n)\chi(-n)\Big|
\Big|\frac1{N}\sum_{n=1}^{N}\bmnu(n)g(j-n) \chi(-n)\Big|d\chi. \quad \quad\quad \quad
\end{eqnarray}
Now, applying Cauchy-Schwarz inequality, we get
\begin{eqnarray*}
 &&\Big|\frac1{N}\sum_{n=1}^{N}\bmnu(n) f(j+n)g(j-n)\Big|^2\\
 &\leq& \Big(\int_{{\widehat{\Z_{\bar{J}}}}} \Big|\sum_{n \in \Z_{\bar{J}}}f(n)\chi(-n)\Big|^2 d\chi\Big) 
 \Big(\int_{\Z_{\bar{J}}} \Big|\frac1{N}\sum_{n=1}^{N}\bmnu(n)g(j-n) \chi(n)\Big|^2 d\chi\Big).
\end{eqnarray*}
Integrating, we see that
\begin{eqnarray}\label{Max-I1}
&&\E_{\Z_{\bar{J}}}\Big( \Big|\frac1{N}\sum_{n=1}^{N}\bmnu(n) f(j+n)g(j-n)\Big|^2\Big) \nonumber\\
&\leq& \|f\|_2^2 \E_{\Z_{\bar{J}}} \Big( \int \Big|\frac1{N}\sum_{n=1}^{N}\bmnu(n)g(j-n) \chi(n)\Big|^2 d\chi\Big) \nonumber\\
&\leq& \|f\|_2^2 \int \E_{\Z_{\bar{J}}} \Big( \Big|\frac1{N}\sum_{n=1}^{N}\bmnu(n)g(j-n) \chi(n) \Big|^2 \Big) d\chi
\end{eqnarray}
We thus need to estimate the RHS of the inequality \eqref{Max-I1}. For that, write
$$ \E_{\Z_{\bar{J}}}\Big(\Big|\frac1{N}\sum_{n=1}^{N}\bmnu(n)g(j-n) \chi(n)\Big|^2\Big)=
\E_{\Z_{\bar{J}}}\Big( \Big|\frac1{N}\sum_{n=1}^{N}\bmnu(n)(U^{-n}g)(j) \chi(n) \Big|^2\Big),$$
where $U$ is the Koopman operator of the shift map $S$. Consequently, by the spectral theorem,  we have
\[
\Big\|\frac1{N}\sum_{n=1}^{N}\bmnu(n)g(S^{-n}j) \chi(n)\Big\|_{\ell^2(\Z_{\bar{J}})}=
\Big\|\frac1{N}\sum_{1 \leq n\leq N} 
\bmnu(n) \lambda^n \chi(n)  \Big\|_{L^2(\sigma_g)},
\]
where $\sigma_g$ is the spectral measure of $g$.\footnote{Recall that $\sigma_g$ 
is a finite  measure on the circle determined by its Fourier transform given by $\widehat{\sigma}_g(n)=<U^ng,g>$.} 
Hence, by Davenport-Hua's estimation~\eqref{vin}, for each $A>0$, we get
\begin{equation}
  \label{eq:Dav2}
  \Big\|\frac1{N}\sum_{n=1}^{N}\bmnu(n)g(S^{-n}j) \chi(n)\Big\|_{\ell^2(\Z_{\bar{J}})} \leq \frac{C_A}{{\big(\log(N)\big)^{A}}}.\|g\|_{\ell^2(\Z_{\bar{J}})},
\end{equation}
where $C_A$ is a constant that depends only on $A$.\\

Combining \eqref{Max-I1} and \eqref{eq:Dav2}, we can rewrite \eqref{Max-I1} as follows
\begin{eqnarray}\label{Max-I2}
\nonumber \E_{\Z_{\bar{J}}}\Big( \Big|\frac1{N}\sum_{n=1}^{N}\bmnu(n) f(j+n)g(j-n)\Big|^2\Big)
\leq  \frac{C_A^2}{{\big(\log(N)\big)^{2A}}}.\|f\|_{\ell^2(\Z_{\bar{J}})}^2.\|g\|_{\ell^2(\Z_{\bar{J}})}^2.,
\end{eqnarray}
Form this, it follows that for any $k=1,\cdots ,K$, we have
$$\sum_{\overset{N_k \leq N \leq  N_{k+1}}{N \in I_\rho}}\E_{\Z_{\bar{J}}}\Big( \Big|\frac1{N}\sum_{n=1}^{N}\bmnu(n) f(j+n)g(j-n)\Big|^2\Big)
$$$$
\leq \sum_{\overset{N_k \leq N \leq  N_{k+1}}{N\in I_\rho}} \frac{C_A^2}{{\big(\log(N)\big)^{2A}}} \|f\|_{\ell^2(\Z_{\bar{J}})}^2
\|g\|_{\ell^2(\Z_{\bar{J}})}^2.$$
Whence
$$\Big\|\sup_{\overset{N_k \leq N \leq  N_{k+1}}{N\in I_\rho}}
\Big|\frac1{N}\sum_{n=1}^{N}\bmnu(n) f(j+n)g(j-n) \Big|\Big\|_{\ell^2(\Z_{\bar{J}})} \leq$$
$$ 
\sqrt{\sum_{\overset{N_k \leq N \leq N_{k+1}}{N \in I_\rho}}\frac{C_A^2}{{\big(\log(N)\big)^{2A}}}}
\|f\|_{\ell^2(\Z_{\bar{J}})}\|g\|_{\ell^2(\Z_{\bar{J}})},$$
Thus by the elementary inequality $\sqrt{a+b} \leq \sqrt{a}+\sqrt{b}, a,b \geq 0$, it follows that
$$\sum_{k=1}^{K}\Big\|\sup_{\overset{N_k \leq N \leq  N_{k+1}}{N\in I_\rho}}
\Big|\frac1{N}\sum_{n=1}^{N}\bmnu(n) f(j+n)g(j-n) \Big|\Big\|_{\ell^2(\Z_{\bar{J}})}$$ $$  
\leq\sum_{k=1}^K\sum_{\overset{N_k \leq N \leq N_{k+1}}{N \in I_\rho}}\frac{C_A}{{\big(\log(N)\big)^{A}}}
\|f\|_{\ell^2(\Z_{\bar{J}})}\|g\|_{\ell^2(\Z_{\bar{J}})},$$
Applying again Cauchy-Schwarz inequality, we conclude that 
$$\sum_{k=1}^{K}\Big\|\sup_{\overset{N_k \leq N \leq  N_{k+1}}{N\in I_\rho}}
\Big|\frac1{N}\sum_{n=1}^{N}\bmnu(n) f(j+n)g(j-n) \Big|\Big\|_{\ell^1(\Z_{\bar{J}})}$$ 
$$\leq C. \sqrt{K}
\|f\|_{\ell^2(\Z_{\bar{J}})}\|g\|_{\ell^2(\Z_{\bar{J}})}.$$
In the same manner we can see from \eqref{Max-I1} and \eqref{eq:Dav2} that we have
$$
\sum_{k=1}^{K}\Big\|\frac1{N_k}\sum_{n=1}^{N_k}\bmnu(n) f(j+n)g(j-n)\Big\|_{\ell^2(\Z_{\bar{J}})}$$ 
$$\leq  
 \sum_{k=1}^{K}\frac{C_A}{\log(N_k)^{A}}
\|f\|_{\ell^2(\Z_{\bar{J}})}\|g\|_{\ell^2(\Z_{\bar{J}})} < C \|f\|_{\ell^2(\Z_{\bar{J}})}\|g\|_{\ell^2(\Z_{\bar{J}})}.$$
Therefore, by the triangle inequality, we get 
$$ \sum_{k=1}^{K}
\Big\|m_{N_k,N_{k+1}}(f,g)\Big\|_{\ell^1(\Z_{\bar{J}})} < C. \sqrt{K}
\big\|f\big\|_{\ell^2(\Z_{\bar{J}})}\big\|g\big\|_{\ell^2(\Z_{\bar{J}})}.$$
Letting ${\bar{J}} \longrightarrow +\infty$, we conclude that we have
$$ \sum_{k=1}^{K}
\Big\|m_{N_k,N_{k+1}}(f,g)\Big\|_{\ell^1(\Z)} < C. \sqrt{K}
\big\|f\big\|_{\ell^2(\Z)}\big\|g\big\|_{\ell^2(\Z)}.$$

and the proof of the theorem is complete.
\end{proof}

\begin{rem}An alternative proof similar to Bourgain's proof can be obtained by using the Fourier transform instead 
of the discrete Fourier 
transform to obtain the same inequalities.
 We recall that the Fourier transform is defined on abelian group $G$ by 
 $$\mathcal{F}(f)(\chi)=\int_{G}f(g)\overline{\chi(g)} dg,$$
where $\chi$ is a character of $G$. For a nice account on the discrete Fourier transform and related topics we refer to \cite{Stein}. 
For the classical Fourier analysis on groups, we refer to \cite{Rudin}.
 \end{rem}

Now, we are able to give the proof of our main result Theorem \ref{Mainofmain}.
\begin{proof}[\textbf{Proof of Theorem \ref{Mainofmain}}.] By a standard argument, we may assume that the map $T$ is ergodic.
Let us assume also that $f,g$ are in $L^\infty(X,\mu)$. Therefore, by Theorem \ref{MaximalI}, it is easily seen that 
\begin{eqnarray}\label{Key-conv}
\frac{1}{K}\sum_{k=1}^{K}\Big\|M_{N_k,N_{k+1}}(f,g)\Big\|_1 \tend{K}{+\infty}0
\end{eqnarray}
Hence, by the same arguments as in \cite{Thouvenot} and \cite{Demeter}, we see that for almost every point $x \in X$, we have
\[
\frac1{[\rho^m]}\sum_{n=1}^{[\rho^m]}\bmnu(n)f(T^nx)g(T^{-n}x)  \tend{m}{+\infty}0,
\]
since the $L^2$-limit is zero by Green-Tao theorem \cite[Theorem 1.1]{GreenT}  combined with Chu's result  \cite[Theorem 1.3]{Chu}.\\

 For the reader's convenience, let us point out that the proof in \cite{Thouvenot} and \cite{Demeter} is obtained by contradiction. Indeed, we assume that the almost everywhere convergence does not hold. Then, we construct an increasing sequence $(N_k)$ for which we establish with the help of the Markov trick that \eqref{Key-conv} can not hold.\\ 

Now, it follows that if $[\rho^m]\leq N < {[\rho^{m+1}]+1}$,  
\begin{eqnarray*}
&&\Big|\frac1{N}\sum_{n=1}^{N}\bmnu(n)f(T^nx)g(T^{-n}x) \Big|\\
&=&\Big| \frac1{N}\sum_{n=1}^{[\rho^m]}\bmnu(n)f(T^nx)g(T^{-n}x) + \frac1{N}\sum_{n=[\rho^m]+1}^{N}\bmnu(n)f(T^nx)g(T^{-n}x)\Big|\\
&\leq & \Big| \frac1{[\rho^m]}\sum_{n=1}^{[\rho^m]}\bmnu(n)f(T^nx)g(T^{-n}x) \Big|+\frac{\big\|f\big\|_{\infty} \big\|g\big\|_{\infty}}{[\rho^m]} (N-[\rho^m]-1)\\
&\leq & \Big|\frac1{[\rho^m]}\sum_{n=1}^{[\rho^m]}\bmnu(n)f(T^nx)g(T^{-n}x)\Big|+\frac{\big\|f\big\|_{\infty} \big\|g\big\|_{\infty}}{[\rho^m]} ([\rho^{m+1}]-[\rho^m]).\\
\end{eqnarray*}
Letting $m$ goes to infinity, we get
\[
\Big| \frac1{[\rho^m}\sum_{n=1}^{[\rho^m]}\bmnu(n) f(T^nx)g(T^{-n}x) \Big| \tend{m}{+\infty}0,
\]
and
\[
\frac{||f||_{\infty}\big\|g\big\|_{\infty}}{[\rho^m]} ([\rho^{m+1}]-[\rho^m])\tend{m}{+\infty}\big\|f\big\|_{\infty} \big\|g\big\|_{\infty}.(\rho-1),
\]
for any $\rho>1$.  Letting $\rho \longrightarrow 1$ we conclude that
\[
\frac1{N}\sum_{n=1}^{N}\bmnu(n)f(T^nx)g(T^{-n}x) \tend{N}{+\infty}0, {\textrm{~~a.e.,}}
\]

To finish the proof, notice that for any $f,g \in L^2(X,\mu)$, and any $\varepsilon>0$, there exist $f_1,g_1 \in L^{\infty}(X,\mu)$ such that
$\Big\|f-f_1\Big\|_2 < \sqrt{\varepsilon},$ and $\Big\|g-g_1\Big\|_2 < \sqrt{\varepsilon}$.
 Moreover, by Cauchy-Schwarz inequality, we have
\begin{eqnarray*}
&&\Big|\frac1{N}\sum_{n=1}^{N}\bmnu(n)(f-f_1)(T^nx)(g-g_1)(T^{-n}x) \Big|\\
&\leq&  \frac1{N}\sum_{n=1}^{N}\big|(f-f_1)(T^nx)\big|\big|(g-g_1)(T^{-n}x)\big|\\
&\leq& \Big(\frac1{N}\sum_{n=1}^{N}\big|(f-f_1)(T^nx)\big|^2\Big)^{\frac12}
\Big(\frac1{N}\sum_{n=1}^{N}\big|(g-g_1)(T^nx)\big|^2\Big)^{\frac12}
\end{eqnarray*}
Applying the ergodic theorem, it follows that for almost all $x \in X$, we have
\[
\limsup_{N \longrightarrow +\infty}\Big|\frac1{N}\sum_{n=1}^{N}\bmnu(n)(f-f_1)(T^nx)(g-g_1)(T^{-n}x) \Big|<  \varepsilon.
\]
Whence, we can write
\begin{eqnarray*}
&&\limsup_{N \longrightarrow +\infty}\Big|\frac1{N}\sum_{n=1}^{N}\bmnu(n)f(T^nx)g(T^{-n}x) \Big|\\
 &\leq&
\limsup_{N \longrightarrow +\infty}\Big|\frac1{N}\sum_{n=1}^{N}
\bmnu(n)f_1(T^nx)g(T^{-n}x)\Big|+
\limsup_{N \longrightarrow +\infty}\Big|\frac1{N}\sum_{n=1}^{N}\bmnu(n)f(T^nx)g_1(T^{-n}x)\Big|\\
&+&\limsup_{N \longrightarrow +\infty}\Big|\frac1{N}\sum_{n=1}^{N}\bmnu(n)f_1(T^nx)g_1(T^{-n}x)\Big|\\
&\leq& \varepsilon+\limsup_{N \longrightarrow +\infty}\Big|\frac1{N}\sum_{n=1}^{N}
\bmnu(n)f_1(T^nx)g(T^{-n}x)\Big|+
\limsup_{N \longrightarrow +\infty}\Big|\frac1{N}\sum_{n=1}^{N}\bmnu(n)f(T^nx)g_1(T^{-n}x)\Big|.
\end{eqnarray*}
We thus need to estimate
$$\limsup_{N \longrightarrow +\infty}\Big|\frac1{N}\sum_{n=1}^{N}
\bmnu(n)f_1(T^nx)g(T^{-n}x)\Big|,$$
and
$$
\limsup_{N \longrightarrow +\infty}\Big|\frac1{N}\sum_{n=1}^{N}\bmnu(n)f(T^nx)g_1(T^{-n}x)\Big|.$$
In the same manner we can see that
\begin{eqnarray*}
&&\limsup_{N \longrightarrow +\infty}\Big|\frac1{N}\sum_{n=1}^{N}
\bmnu(n)f_1(T^nx)(g-g_1)(T^{-n}x)\Big|\\
&\leq& \limsup_{N \longrightarrow +\infty}\Big(\frac1{N}\sum_{n=1}^{N}|f_1(T^nx)|^2\Big)^{\frac12}
\limsup_{N \longrightarrow +\infty}\Big(\frac1{N}\sum_{n=1}^{N}|(g-g_1)(T^nx)|^2\Big)^{\frac12}\\
&\leq& \big\|f_1\big\|_2 \big\|g-g_1\big\|_2\\
&\leq& \Big(\big\|f\big\|_2+\sqrt{\varepsilon}\Big). \sqrt{\varepsilon}
\end{eqnarray*}
This gives
\begin{eqnarray*}
&&\limsup_{N \longrightarrow +\infty}\Big|\frac1{N}\sum_{n=1}^{N}
\bmnu(n)f_1(T^nx)g(T^{-n}x)\Big|\\
&\leq& \Big(\Big\|f\Big\|_2+\sqrt{\varepsilon}\Big). \sqrt{\varepsilon}+
\limsup_{N \longrightarrow +\infty}\Big|\frac1{N}\sum_{n=1}^{N}
\bmnu(n)f_1(T^nx)g_1(T^{-n}x)\Big|\\
&\leq& \Big(\big\|f\big\|_2+\sqrt{\varepsilon}\Big). \sqrt{\varepsilon}+0
\end{eqnarray*}
Summarizing, we obtain the following estimates
\begin{eqnarray*}
&&\limsup_{N \longrightarrow +\infty}\Big|\frac1{N}\sum_{n=1}^{N}\bmnu(n)f(T^nx)g(T^{-n}x) \Big|\\
 &\leq& \varepsilon+\Big(\big\|f\big\|_2+\sqrt{\varepsilon}\Big). \sqrt{\varepsilon}+\Big(\big\|g\big\|_2+\sqrt{\varepsilon}\Big). \sqrt{\varepsilon}
\end{eqnarray*}
Since $\varepsilon>0$ is arbitrary, we conclude that for almost every $x \in X$,
\[
\frac1{N}\sum_{n=1}^{N}\bmnu(n)f(T^nx)g(T^{-n}x) \tend{N}{+\infty}0.
\]
This complete the proof of the theorem.
\end{proof}

As a consequence of our proof, we have proved Theorem \ref{TDB}. Indeed, by taking $\bmnu=1$ in equations \eqref{Fourier} and \eqref{FourierII} ( see also Equation (2.15) in \cite{Bourgain-D}), we have, for any $j \in \Z_{\bar{J}}$,  
\begin{align}\label{Bequa:1}
m_{N_k,N_{k+1}}(f,g)(j)
\leq \int_{{\widehat{\Z_{\bar{J}}}}} \big|\F(f)(\chi) \big|.m_{N_k, N_{k+1}}(g_{\chi},1_{\bar{J}})(j) d\chi,
\end{align}
where $g_{\chi}(x)=g(x)\chi(x),$ for any $x \in \Z_{\bar{J}}.$ Since, for any $M_0,M_1 \in \N$, 
$$\F(G_{M_1,j})-\F(G_{M_0,j})=\F\big(G_{M_1,j}-G_{M_0,j}\big),$$
and
\begin{align}\label{BB:3}
\Big|\frac{1}{M_1}\sum_{n=1}^{M_1}f(j+n)g(j-n)
-\frac{1}{M_0}\sum_{n=1}^{M_0}f(j+n)g(j-n)\Big| \nonumber\\
\leq \int_{{\widehat{\Z_{\bar{J}}}}} 
\big|\F(f)(\chi)\big| \big|\F\big(G_{M_1,j}-G_{M_0,j}\big)\big|d\chi. 
\end{align}
Notice that 
$$\Big|\frac{1}{M_1}\sum_{n=1}^{M_1}g(j-n) \chi(-n)\Big|=
\Big|\frac{1}{M_1}\sum_{n=1}^{M_1}g(j-n) \chi(j-n)\Big|.$$
 Let us further notice that, obviously, if $F \in \ell^2(\Z)$, then $F_{\chi} \in \ell^2(\Z)$, with $\big\|F_\chi\|_2=\big\|F\|_2.$ \\

Now, as before,
integrating and applying Cauchy-Schwarz inequality combined with Parseval inequality, we obtain
\begin{align}\label{Bequa:2}
\E_{\Z_{\bar{J}}}\Bigg(m_{N_k,N_{k+1}}(f,g)\Bigg)
&\leq  \int_{{\widehat{\Z_{\bar{J}}}}} \big|\F(f)(\chi) \big|.\E_{\Z_{\bar{J}}}\Big(m_{N_k, N_{k+1}}(g_{\chi},1_{\bar{J}})(j)\Big) d\chi \nonumber\\
&\leq\Bigg(\int  \big|\F(f)(\chi) \big|^2 d\chi\Bigg)^{\frac12} 
\Bigg( \int_{{\widehat{\Z_{\bar{J}}}}}  \E_{\Z_{\bar{J}}}\Big(m_{N_k, N_{k+1}}(g_{\chi},1_{\bar{J}})(j)\Big)^2 d\chi\Bigg)^{\frac12}\\
&\leq \big\|f\big\|_{\ell^2(Z_J)}^2.\Bigg(\int_{0}^{1} \E_{\Z_{\bar{J}}}\Big(\big(m_{N_k, N_{k+1}}(1,g_{\chi})\big)^2\Big) d\chi\Bigg)^{\frac12}\nonumber,
\end{align}
The last inequality follows from the Parseval inequality.
\noindent{}We further have 
%\begin{align}\label{Bequa:4}
% \sum_{k=1}^{+\infty}\E_{\Z_{\bar{J}}}\Big(\big(m_{N_k, N_{k+1}}(\1_{\bar{J}},g_{\chi})\big)^2\Big) 
% &=\sum_{k=1}^{+\infty}\Big\|m_{N_k, N_{k+1}}(\1_{\bar{J}},g_{\chi})\Big\|_{\ell^2(\Z_{\bar{J}})}^2 \nonumber\\
 %&\leq C_\rho \|g_\chi\|_{\ell^2(\Z_{\bar{J}})}^2,
%\end{align}
\begin{align}\label{Bequa:4}
\sum_{k=1}^{K}\E_{\Z_{\bar{J}}}\Big(\big(m_{N_k, N_{k+1}}(\1_{\bar{J}},g_{\chi})\big)^2\Big) 
&=\sum_{k=1}^{K}\Big\|m_{N_k, N_{k+1}}(\1_{\bar{J}},g_{\chi})\Big\|_{\ell^2(\Z_{\bar{J}})}^2 \nonumber\\
&\leq A(K)^2 \|g_\chi\|_{\ell^2(\Z_{\bar{J}})}^2,
\end{align}
where $A(K)=\sqrt[4]{K}$, since, by \cite[(3) in the proof of Theorem 3.]{Thouvenot}, for any $F \in \ell^2(\Z)$, we have \footnote{Let us point out that here we consider only the linear case 
	$\Big(\frac{1}{N}\sum_{n=1}^{N})f(n+x)\Big)$ rather that the polynomials case $\Big(\frac{1}{N}\sum_{n=1}^{N})f(n+x^d)\Big)$, $ d \geq 2$. Therefore, it is a simple exercise to see that  for any $F \in \ell^2(\Z)$, we have
	$$\sum_{k=1}^{+\infty}\Big\|\Big(\big(m_{N_k, N_{k+1}}(\1,F)\big)\Big\|_2 \leq C_\rho \big\|F\big\|_2,$$

(see for instance, %by (5.11) from \cite[Proof of Theorem 5.2]{elabdal-Arxiv} or 
\cite[ Corollaire.6.4.3]{Weber})}
$$\sum_{k=1}^{K}\Big\|\Big(\big(m_{N_k, N_{k+1}}(\1,F)\big)\Big\|_2^2 \leq \sqrt{K} \big\|F\big\|_2,$$
%where $\frac{A(K)}{K} \longrightarrow 0$ as $K \longrightarrow +\infty$.  Indeed, this implies
%\begin{align*}
%&\sum_{k=1}^{K}\Big\|\Big(\big(m_{N_k, N_{k+1}}(\1,F)\big)\Big\|_2^2\\
%&\leq \sum_{k=1}^{K}\Big\|\Big(\big(m_{N_k, N_{k+1}}(\1,F)\big)\Big\|_2 \sup_{ 1  \leq k \leq K } \Big\|\Big(\big(m_{N_k, N_{k+1}}(\1,F)\big)\Big\|_2\\
% &\leq  A(K)^2 \big\|F\big\|_2^2.
%\end{align*}
Whence, again by Cauchy-Schwarz inequality
\begin{align}\label{eqf1}
\sum_{k=1}^{K}\Big\|m_{N_k,N_{k+1}}(f,g)\Big\|_{\ell^1(\Z_{\bar{J}})}
&\leq \sqrt{K}\Big(\sum_{k=1}^{K}\Big\|m_{N_k,N_{k+1}}(f,g)\Big\|_{\ell^2(\Z_{\bar{J}})}^2\Big)^
{\frac12}\\
&\leq \sqrt[4]{K^3} \big\|f\big\|_{\ell^2(\Z_{\bar{J}})} \big\|g\big\|_{\ell^2(\Z_{\bar{J}})},
\end{align}
%where $A(K)=C_\rho \sqrt{K}.$
Letting $\bar{J} \longrightarrow +\infty$, we obtain 
\begin{align}\label{eqf2}
\sum_{k=1}^{K}\Big\|m_{N_k,N_{k+1}}(f,g)\Big\|_{\ell^1(\Z)}
\leq \sqrt[4]{K^3} \big\|f\big\|_{\ell^2(\Z)} \big\|g\big\|_{\ell^2(\Z)}.
\end{align}
Applying now the same arguments as in the proof of Theorem \ref{Mainofmain} the desired result follows.
\begin{rem}Notice that our proof yields that the convergence almost sure holds for the short interval. Thanks to Zhan's estimation (equation \eqref{Zhantao}). Let us notice also that an alternative proof to Theorem 2.1 can obtained by using the spectral regularization principal in \cite[Chap. 6]{Weber} since the shift map on $\Z$ has a simple Lebesgue spectrum.
\end{rem}
We end this section by stating the following conjecture.
\begin{xconj} Let $\bmnu$ be a aperiodic bounded multiplicative function and $l \geq 2$ a positive integer.
If $T_1 ,\cdots,T_l$ are commuting measure preserving transformations acting on the same probability space $(X,\A,\mu)$, then for all $f_1 ,\cdots,f_l \in L^\infty(X,\mu)$, for almost all $x \in X$, we have
$$\frac1{N}\sum_{n=1}^{N}\bmnu(n) \prod_{j=1}^{l}f_j({T_j}^{n}x)\tend{N}{+\infty}0.$$
\end{xconj}
We remind that  $\bmnu$ is a aperiodic multiplicative function if
$$
\left\{
  \begin{array}{ll}
   \bmnu(mn)=\bmnu(m)\bmnu(n), & \hbox{for all $m,n \in \N$ such that $m \wedge n=1$;~and}  \\
    \ds \frac1N\sum_{n=1}^{N}\bmnu(an+b) \tend{N}{+\infty}0 , & \hbox{for all $a,b \in \N^* \times \N$..}
  \end{array}
\right.
$$
\begin{thank}
The author wishes to express his thanks to XiangDong Ye and Benjamin Weiss
for a stimulating conversations on the subject. He is also thankful to
Nalini Anantharaman and to university of Strasbourg, IRMA, where
a part of this paper was written. The author
wishes also to thank Wilfrid Gangbo and the university of UCLA where the paper was
revised, for the invitation and hospitality. The author wishes further to express his 
cordial thanks to H. Daboussi  for bringing to his attention his paper \cite{DaboussiII} related to Theorem \ref{KBSZ}.  
\end{thank}

\end{document}